\documentclass[12pt]{amsart}
\usepackage{amsfonts}
\usepackage{amsmath}
\usepackage{amsxtra}
\usepackage{amssymb,latexsym}
\usepackage[mathcal]{eucal}
\usepackage{amscd}
\usepackage[pdftex,bookmarks,colorlinks,breaklinks]{hyperref}

\usepackage{epsfig}
\usepackage{graphicx}
\usepackage{color}
\usepackage{verbatim}

\input xy
\xyoption{all}

\newcommand{\Fcal}{\mathcal{F}}
\newcommand{\Gcal}{\mathcal{G}}
\newcommand{\Hcal}{\mathcal{H}}

\newcommand{\Rcal}{\mathcal{R}}

\newcommand{\ch}{\mathbf{1}}

\newcommand{\R}{\mathbb{R}}

\newcommand{\N}{\mathbb{N}}
\newcommand{\T}{\mathbb{T}}

\newcommand{\im}{\mathbf{\mathfrak{m}}}

\newcommand{\al}{\alpha}
\newcommand{\Ga}{\Gamma}
\newcommand{\ga}{\gamma}
\newcommand{\del}{\delta}
\newcommand{\Del}{\Delta}
\newcommand{\ep}{\epsilon}

\newcommand{\la}{\lambda}
\newcommand{\La}{\Lambda}

\newcommand{\br}{\vspace{3 mm}}
\newcommand{\imp}{\Rightarrow}

\newcommand{\id}{{\rm{id}}}
\newcommand{\Id}{{\rm{Id}}}
\newcommand{\diam}{{\rm{diam\,}}}
\newcommand{\supp}{{\rm{supp\,}}}

\newtheorem{thm}{Theorem}[section]

\newtheorem{lem}[thm]{Lemma}
\newtheorem{prop}[thm]{Proposition}

\theoremstyle{definition}
\newtheorem{defn}[thm]{Definition}

\newtheorem{rmk}[thm]{Remark}

\newtheorem{rems}[thm]{Remarks}

\theoremstyle{remark}

\newtheorem{que}[thm]{Question}

\numberwithin{equation}{section}

\def \N {\mathbb N}

\def \R {\mathbb R}

\def \sq {sequence}
\def \diam {\mathsf{diam}}
\def \cob {\mathsf{Cob}}
\def \tl {topological}
\def \im {invariant measure}

\def \ds {dynamical system}

\def \id {\mathsf{Id}}

\def \xt {$(X,T)$}
\def \xmt {$(X,\mu,T)$}
\def \zs {$(Z,S)$}
\def \zns {$(Z,\nu,S)$}

\def \Id {\mathsf{Id}}

\begin{document}
\title{Isomorphic extensions and applications}
\author{Tomasz Downarowicz and Eli Glasner}

\address{Institute of Mathematics\\
     Polish Academy of Science\\
         \'Sniadeckich 8, 00-656 Warsaw\\
         Poland}
\email{downar@pwr.edu.pl}

\address{Department of Mathematics\\
     Tel Aviv University\\
         Tel Aviv\\
         Israel}
\email{glasner@math.tau.ac.il}

\date{February 6, 2015}

\begin{abstract}
If $\pi:(X,T)\to(Z,S)$ is a \tl\ factor map between uniquely ergodic \tl\ \ds s, then \xt\ is called an isomorphic extension of \zs\ if $\pi$ is also a measure-theoretic isomorphism.
We consider the case when the systems are minimal and we pay special attention to equicontinuous \zs.
We first establish a characterization of this type of isomorphic extensions in terms of mean equicontinuity,
and then show that an isomorphic extension need not be almost one-to-one,
answering questions of Li, Tu and Ye.
\end{abstract}



\thanks{The research of the first named author is supported by the NCN (National Science Center, Poland) grant 2013/08/A/ST1/00275. The research of the second named author was supported by a grant of ISF 668/13}

\maketitle

%
\section{Introduction}
Throughout, by a \tl\ \ds\ (denoted \xt, or similarly) we will always mean the action of a homeomorphism 
$T$ on an infinite compact metric space $X$. Although many of our results apply to the noninvertible case, 
for simplicity, we focus on invertible systems only.
\smallskip

By a \emph{\tl\ model} of an (invertible) ergodic measure-preserving transformation (\emph{m.p.t.} for short) $(\Omega, \Fcal, \nu, S)$, we will mean any uniquely ergodic \tl\ system \xt\ (with the unique \im\ $\mu$), such that $(X,\mathsf{Borel}(X), \mu,T)$ and $(\Omega, \Fcal, \nu, S)$ are measure-theoretically isomorphic. 

The celebrated Jewett-Krieger theorem asserts that every invertible ergodic m.p.t. has a strictly ergodic (i.e. minimal and uniquely ergodic) \tl\ model. In this work, we are interested in the relations between various 
strictly ergodic models of the same ergodic system. More specifically we will focus on the situation when one model is a \tl\ extension of another. Also, of particular interest to us is the case when the 
underlying m.p.t. belongs to the class of the simplest ergodic systems, namely those with discrete spectrum. Recall that a ``standard model'' of a discrete spectrum ergodic system has the form of a rotation, by a \tl\ generator, on a compact monothetic group or, equivalently, is a minimal equicontinuous system (Halmos--von Neumann Theorem, see e.g. \cite[Chapter 1, Section 2]{G-03}).

We will provide a natural classification of all the 
strictly ergodic
topological models of an ergodic system with discrete spectrum which topologically extend the standard model. 
Before we formulate our results, we need to establish some basic terminology. We will assume that the reader is familiar with the textbook notions of a factor (in particular, maximal equicontinuous factor), extension, isomorphism, ergodicity, minimality, discrete spectrum, etc., as well as with the definition of  lower and upper Banach density of a subset of integers and some related notions. When dealing with an m.p.t.\! which arises from a \tl\ \ds\ equipped with an invariant probability measure, we will always assume that the sigma-algebra in question is the Borel sigma-algebra completed with respect to that measure (and skip it in the notation of the system). The reader is referred to Furstenberg's classical monograph \cite{Fur-81}, or to \cite{G-03} for more details and background.

\medskip

Most of the time we will consider a pair of uniquely ergodic \tl\ \ds s \xt\ and \zs. In such a case, $\mu$ 
and $\nu$ will always denote the unique \im s on $X$ and $Z$, respectively. 
Notice that a \tl\ factor of a uniquely ergodic system is uniquely ergodic as well.

\begin{defn}
We say that \xt\ is an \emph{isomorphic extension} of \zs\ if \xt\ is uniquely ergodic and there 
exists a \tl\ factor map $\pi:X\to Z$ which is, at the same time, a measure-theoretic isomorphism 
between \xmt\ and \zns.
\end{defn}

It is clear that an isomorphic extension of a \tl\ model of some (ergodic) m.p.t. is another \tl\ model of the same m.p.t. Notice that the requirement in the definition is stronger than just assuming that \xt\ is a \tl\ extension of \zs\ and that the systems \xmt\ and \zns\ are measure-theoretically isomorphic. The measure-theoretic isomorphism must be realized by the same map which establishes the \tl\ factor relation. Isomorphic extensions are important because they carry over many ``hybrid properties'' of the base system \zs\ to the extended system \xt. The adjective ``hybrid'' refers to properties which combine measure-theoretic and \tl\ notions. As an example recall the notion of uncorrelation: two bounded complex-valued \sq s $(x_n)_{n\ge 1}$ and $(y_n)_{n\ge 1}$ are called \emph{uncorrelated} if
$$
\lim_n\left|\frac1n\sum_{i=1}^nx_i\bar y_i - \Bigl(\frac1n\sum_{i=1}^nx_i\Bigr) \Bigl(\frac1n\sum_{i=1}^n \bar y_i\Bigr)\right|= 0.
$$
\emph{Sarnak's conjecture} asserts that any \tl\ dynamical system \xt\ with zero \tl\ 
entropy satisfies the following \emph{uncorrelation condition}: 
\begin{quote} 
Fix an arbitrary continuous (complex-valued) function $f$ on $X$ and any point $x\in X$. Then, the \sq\ $(a_n)_{n\ge 1}$ defined by: $a_n = f(T^nx)$ is uncorrelated to the \emph{M\"obius function} $(\mu_n)_{n\ge 1}$ (see \cite{Sarnak} for more details).
\end{quote} 
In \cite[Theorem 4.1]{DK} it is proved that if a \tl\ system \zs\ fulfills 
the uncorrelation condition and \xt\ is an isomorphic extension of \zs\ then \xt\ also fulfills the uncorrelation condition. 
In fact, the proof does not rely on any specific property of the M\"obius function; it shows that isomorphic extensions preserve the property of uncorrelation with respect to any fixed bounded complex-valued \sq.

Since minimal rotations of compact monothetic groups fulfill the uncorrelation condition, it follows from \cite{DK} that every isomorphic extension of the standard model of an ergodic system with discrete spectrum also fulfills the uncorrelation condition\footnote{Let us mention that recently, H. El Abdalaoui, M. Lema\'nczyk and T. de la Rue (private communication) proved that any \tl\ model of an ergodic system with {\sl irrational} discrete spectrum satisfies Sarnak's uncorrelation condition. 
It is striking that even for the system with the rational discrete spectrum $\{-1,1\}$ (whose standard model is the two-point periodic system) the validity of the conjecture for all (uniquely ergodic) \tl\ models remains undecided.}.

\medskip
An isomorphic extension can be described as one given by a \tl\ factor map $\pi:X\to Z$ which becomes invertible after discarding from both spaces subsets of measure zero. By analogy, one can require that the map $\pi$ becomes invertible after discarding from both spaces some meager (first category) sets. This leads to the notion of an almost 1-1 extension, which can be regarded as a \tl\ analog of an isomorphic extension (how good is this analogy -- will become clear from what follows):

\begin{defn}
We say that \xt\ is an \emph{almost one-to-one (almost 1-1) extension } of \zs\ if there exists a \tl\ factor map $\pi:X\to Z$ such that the union of the singleton fibers is dense in $X$. If \xt\ is minimal it suffices that at least one singleton fiber exists.
\end{defn}

By invariance, for any ergodic measure on $X$, the union of singleton fibers has measure either zero or one.
Thus, uniquely ergodic almost 1-1 extensions can be classified as follows (note that these are ``hybrid'' notions):

\begin{defn}
Let \xt\ be a uniquely ergodic system which is an almost 1-1 extension of \zs. We say that this extension
is \emph{regular} if the union of the singleton fibers has full $\mu$ measure. 
Otherwise, when this union has measure zero, the extension is called \emph{irregular}.
\end{defn}

It is obvious that a regular almost 1-1 extension is automatically isomorphic. On the other hand, it is well-known that irregular almost 1-1 extensions are usually far from being isomorphic and many invariants
need not be preserved. For instance the extension may have much richer spectral properties, larger entropy, many more \im s, etc. So, two natural questions arise: can an irregular almost 1-1 extension still be isomorphic? And, can an isomorphic extension be not almost 1-1 at all? The first question has a positive answer; an appropriate example is provided in \cite[Example 5.1]{DK}. Another example of this phenomenon is to be found in the work of Kerr and Li \cite[Section 11]{KL}, see Remark 5.2.\! in \cite{Gl-str}.
In the present paper we will provide a positive answer to the second question, as well,
already in the class of minimal systems. 
As a result, we obtain the following classification of isomorphic extensions (ordered decreasingly with respect to ``\tl\ closeness''), where all 
three, mutually disjoint classes, are nonempty:
\begin{enumerate}
	\item regular almost 1-1 extensions,
	\item irregular, yet isomorphic, almost 1-1 extensions,
	\item isomorphic, but not even almost 1-1 extensions.
\end{enumerate}

For example, if \zs\ is a standard model of an ergodic m.p.t.\! with 
infinite
discrete spectrum, we obtain three types of 
strictly ergodic
\tl\ models with subtly different \tl\ properties.

Let us mention that there exist 
strictly ergodic \tl\ models of ergodic systems with 
infinite 
discrete spectrum which are \tl ly much removed from the standard models. For instance, there exist such models which are \tl ly weakly (or even srtongly) mixing, in which case there is no topological factor map in either direction, see e.g. \cite{Leh} or \cite{GW-06}.

\medskip
It turns out that isomorphic extensions can also be used to characterize the classes of systems called \emph{mean equicontinuous} and \emph{Weyl mean equicontinuous}.

\begin{defn}
If \xt\ is a \ds, we define the \emph{Besicovitch} and \emph{Weyl distances} between points, respectively, as 
\begin{align*}
d_{\mathsf B}(x,y): &=\limsup_{n\to\infty}\frac1n \sum_{i=0}^{n-1}d(T^ix,T^iy), \\ 
d_{\mathsf W}(x,y): &=\limsup_{n-m\to\infty}\frac1{n-m} \sum_{i=m}^{n-1}d(T^ix,T^iy).
\end{align*}
A system \xt\ is called \emph{mean equicontinuous} or \emph{Weyl mean equicontinuous}\footnote{In \cite{LTY} the authors call this notion \emph{Banach mean equicontinuous} (probably referring to 
the notion of upper Banach density). We believe that the reference to Weyl is more appropriate here.} 
if for every $\epsilon>0$ there exists $\delta>0$ such that 
$$
d(x,y)<\delta \implies d_{\mathsf B}(x,y)<\epsilon, \ \ d(x,y)<\delta \implies d_{\mathsf W}(x,y)<\epsilon,
$$
respectively.
\end{defn}
\medskip

It is easy to see that the Besicovitch and Weyl distances are invariant pseudometrics. Mean equicontinuity is continuity of the quotient map $(X,d)\to(X/_\approx\,,\,d_{\mathsf B})$, where $\approx$ is the \emph{mean asymptotic} relation $x\approx y \iff d_{\mathsf B}(x,y)=0$. 
An analogous statement with $d_{\mathsf B}$ replaced by $d_{\mathsf W}$ applies to Weyl mean equicontinuity. 
\smallskip

Li, Tu and Ye  \cite{LTY} posed the following questions, 
based on observation of the examples which were available to them:
\begin{que}\label{que0}
Is every minimal mean equicontinuous system also Weyl mean equicontinuous?
\end{que}
\begin{que}\label{que}
Is every minimal Weyl mean equicontinuous system an almost 1-1 extension of its maximal equicontinuous factor?
\end{que}
In the next two sections we prove theorems which allow us to answer Question \ref{que0} positively (Theorem \ref{mean}), then we reduce Question \ref{que} to a problem concerning isomorphic extensions, and finally, 
answer this question negatively (Theorem \ref{main}).
In the last section we present some concrete instances of our main result and refer to some related
topics in the literature.

\br

\section{A characterization of minimal mean equicontinuous systems}

\begin{thm}\label{mean}
Let \xt\ be a minimal \tl\ \ds\ with maximal equicontinuous factor \zs.
Then the following are equivalent
\begin{enumerate}
	\item \xt\ is mean equicontinuous,
	\item \xt\ is Weyl mean equicontinuous,
  \item \xt\ is an isomorphic extension of \zs\ (in particular, \xt\ is then uniquely ergodic).
\end{enumerate}
\end{thm}

\begin{rems}
\begin{enumerate}
	\item[a.] Note that if \xt\ is an isomorphic extension of a minimal equicontinuous system \zs, then \zs\ is necessarily its maximal equicontinuous factor.
	\item[b.] Huang, Lu and Ye \cite{HLY} introduce a hybrid notion of $\mu$-equicontinuity (of a \tl\ system with a fixed \im\ $\mu$) and prove that this property implies discrete spectrum of \xmt. Garcia-Ramos \cite{GR} proves that an even weaker (hybrid) property, $\mu$-mean equicontinuity, is 
in fact equivalent to discrete spectrum of \xmt. The implication 1.$\implies$3. above (which is also implicit in the proof of Theorem 3.8 of \cite{LTY}) has the same flavor; it implies that mean equicontinuous systems are uniquely ergodic with discrete spectrum. Clearly, the condition 3. gives 
more specific information about the system.
\end{enumerate}
\end{rems}

\begin{proof}[Proof of the theorem] 
We let $\pi:X\to Z$ denote the maximal equicontinuous factor map.
\medskip

$1. \imp 3.$:\ 
Let \xt\ be mean equicontinuous. First we show that two points which are in the same fiber
over the maximal equicontinuous factor are mean asymp\-totic (i.e., the Besicovitch distance 
between them is zero).
 
By a theorem of Veech \cite{V} 
(see also \cite{Aus}) such points, say $x,y$, are regionally proximal, i.e., there are \sq s $n_k,\, x_k,\, y_k$ with $x_k\to x$, $y_k\to y$ and $d(T^{n_k}x_k, T^{n_k}y_k)\to 0$.
Given $\epsilon > 0$ let $\delta$ be as in the definition of mean equicontinuity. For 
a sufficiently large $k$ we have:
\begin{align*}
&d(x_k,x)<\delta, \text{ hence } d_{\mathsf B}(x_k,x)<\epsilon, \\
&d(y_k,y)<\delta, \text{ hence } d_{\mathsf B}(y_k,y)<\epsilon, \\
&d(T^{n_k}x_k,T^{n_k}y_k)<\delta, \text{ hence } d_{\mathsf B}(x_k,y_k) = d_{\mathsf B}(T^{n_k}x_k,T^{n_k}y_k)<\epsilon.
\end{align*}
We have shown that $d_{\mathsf B}(x,y)<3\epsilon$ for every $\epsilon>0$, hence $d_{\mathsf B}(x,y)=0$. 

As \zs\ is minimal and equicontinuous, it is strictly ergodic. 
Suppose now that $\pi$ is not an isomorphic extension. This means that on $X$ there are either (at least) two distinct ergodic measures $\mu$, $\mu'$, or there is just one ergodic measure $\mu$, but its disintegration over $\nu$ produces fiber measures which, with positive $\nu$ probability, are not 
point masses.

In the first case we take any joining $\xi$ of $\mu$ with $\mu'$ over the common factor $\nu$, in the other case we take the relatively independent self-joining of $\mu$ with itself over the common factor $\nu$ 
(we set $\mu'=\mu$). In either case $\xi$ is not concentrated on the diagonal. Becuase both $\mu$ and $\mu'$ are ergodic, $\xi$ decomposes into ergodic joinings. In either case we can find an ergodic component $\xi'$ of $\xi$ which not concentrated on the diagonal. The 
push-forward measure $(\pi \times \pi)_*(\xi)$ and hence also $(\pi \times \pi)_*(\xi')$ is the identity self-joining of $\nu$ (this is the meaning of the fact that our joinings were ``over the common factor $\nu$''). Since $\xi'$-almost every pair $(x,y)$ is generic for $\xi'$, and $\xi'$-almost every pair $(x,y)$ satisfies $\pi x=\pi y$, we can find a pair $(x,y)$ generic for $\xi'$, such that $x$ and $y$ are in the same fiber. On the other hand, since $\xi'$ is not supported by the diagonal, there is a closed $\epsilon$-neighborhood of the diagonal whose open complement has positive $\xi'$ measure, say $\gamma>0$.
As $(x,y)$ is generic, the orbit of $(x,y)$ visits this complement with lower density larger than $\gamma$. Thus, $d_{\mathsf B}(x,y)>\epsilon\gamma>0$, and this contradicts the fact that the pair $x,y$ is mean 
asymptotic.

\medskip

$3. \imp 2.$:\ 
Now suppose that \xt\ is a (uniquely ergodic) isomorphic extension of \zs. 
Let $\pi$ be the factor map from $X$ onto $Z$ (which is also a measure-theoretic isomorphism). On $Z$ we choose a metric $d$ for which $S$ is an isometry. Moreover, we can assume that the metric on $X$ (also denoted by $d$) satisfies $d(x,y)\ge d(\pi x,\pi y)$.

The map $\pi$ becomes invertible after discarding a null set in $X$ and a null set in $Z$. 
The inverse function $\pi^{-1}: Z \to X$ (defined almost everywhere on $Z$) is measurable, hence,
by Luzin's theorem, continuous when restricted to some subset $A\subset Z$ of large $\nu$ measure, say $\nu(A)>1-\epsilon$. By the regularity of $\nu$ we can assume that $A$ is closed, hence compact. 
Let $f : A \to X$ denote the restriction $\pi^{-1}|_A$. Let $U$ be the $\delta$-neighborhood of $f(A)$, where $\delta$ will be specified later. Clearly, $\mu(U)\ge\mu(f(A))=\nu(A)>1-\epsilon$. 

Take two points $x,y\in X$ with $d(x,y)<\delta$ and observe their forward orbits. First note
that  $d(\pi x, \pi y) < \delta$ and, because the map $S$ is an isometry, the corresponding points on 
their orbits remain at the same distance. By the unique ergodicity of \xt, the points $x$ and $y$ are 
uniformly generic\footnote{In a uniquely ergodic system \xt\ every point $x$ is \emph{uniformly generic}, which means that the limit \vspace{-5pt} $$\lim_{n-m\to\infty}\frac1{n-m}\sum_{i=m}^{n-1}f(T^ix)$$ exists for every continuous function $f$ and equals the integral of $f$.} for $\mu$. In particular, their forward orbits visit $U$ with lower Banach densities larger than $1-\epsilon$. Hence, with lower Banach density at least $1-2\epsilon$ they fall simultaneously in $U$.

Consider an $n$ such that both $T^nx, T^ny$ fall in $U$. Then, there are points $x',y'\in f(A)$ 
such that $d(T^nx,x')<\delta, \ d(T^ny,y')<\delta$. This implies that 
\begin{multline*}
d(\pi x',\pi y')\le d(\pi x',\pi T^nx) + d(\pi T^nx,\pi T^ny) + d(\pi T^ny,\pi y')\le \\
d(x',T^nx) + d(S^n \pi x,S^n \pi y) + d(T^ny, y') \le 3\delta.
\end{multline*}
Since $x'=f(\pi x'), \ y'=f(\pi y')$, and $f$ is uniformly continuous on $A$, for a suitably small $\delta<\frac\epsilon4$ we have $d(x',y')<\frac\epsilon2$, and thus
$$
d(T^nx,T^ny)\le d(T^nx,x')+d(x',y')+d(y',T^ny)<\delta+\frac\epsilon2+\delta<\epsilon.
$$
We have shown that given $\epsilon >0$ there exists $\delta > 0$ such that $d(x,y)<\delta$ implies
$d(T^nx, T^ny)<\epsilon$ for $n$'s in a set of lower Banach density at least $1-\epsilon$. 
But that is exactly the meaning of Weyl mean equicontinuity of \xt. 
\medskip

The implication $2. \imp 1.$ is trivial.
\end{proof}

The above theorem not only answers Question \ref{que0}, but also allows us to formulate Question \ref{que} 
in an equivalent way which does not refer to the notion of mean equicontinuity:

\begin{que}\label{que1} Let $(X,T)$ be a minimal system and let \zs\ be its maximal equicontinuous factor via a map $\pi:X\to Z$. Suppose that \xt\ is an isomorphic extension of \zs, is $\pi$ necessarily almost one-to-one?
\end{que}

As already mentioned, we will answer this question negatively. In fact, in the next section we will prove a much more general result, where isomorphic non almost one-to-one extensions will be shown to be 
generic\footnote{Here and in the sequel, the term ``generic'' refers to ``belonging to a residual subset'', where a residual subset of a Polish space $X$ is a set which contains a dense $G_\del$ subset of $X$.
Not to be confused with ``generic points'' for an \im.} in a certain setup.
\medskip

We take this opportunity to investigate the minimal size of a fiber in a \tl\ extension. A priori, the general semicontinuity properties of the ``fiber size function'' do not suffice to draw the conclusion 
that in the non-almost 1-1 case the infimum of this function is 
positive\footnote{The function $\diam$ defined in the proof is upper (and not lower) semicontinuous, thus it need not achieve its minimum.}, hence the lemma below may be of interest.

\begin{lem}
Let $\pi:(X,T)\to (Z,S)$ be a \tl\ factor map, where $X$ (and hence also $Z$) is minimal.
Define $\diam:Z\to[0,\diam(X)]$ by $\diam(z)=\diam(\pi^{-1}(z))$. Then either $\pi$ is 
almost 1-1 (in which case $\inf(\diam) = \min(\diam)=0$) or $\inf(\diam)>0$. 
\end{lem}
\begin{proof}
Suppose $\pi$ is not almost 1-1. Then $\diam$ is positive everywhere on $Z$. 
Pick a decreasing to zero \sq\ $(\delta_n)$, such that $d(x,x')<\delta_{n+1}\implies d(Tx,Tx')<\delta_n$ (here $d$ represents the metric on $X$).
Because $S$ is a homeomorphism, we have $z=S^{-1}Sz$ and hence 
$$
\pi^{-1}(z) = \pi^{-1}S^{-1}Sz = T^{-1}\pi^{-1}Sz,
$$
which implies that each set $A_{n+1}=\{y:\diam(y)<\delta_{n+1}\}$ is mapped by $S$ into $A_n$. Notice that 
all the sets $A_n$ are open. If all these sets were nonempty, then each of them would contain arbitrarily
long pieces of orbits, which, by minimality, would imply that all these sets were dense. By the Baire 
category theorem their intersection would be nonempty, implying that $\diam(z)=0$ at some point, i.e., the extension would be almost 1-1. We have shown that indeed $\diam$ is bounded from below by some positive constant. 
\end{proof}

We now prove another characterization of isomorphic extensions which will become useful in the following section.
Given a factor map between dynamical systems $\pi : (X,T) \to (Z,S)$, we form an associated
system, $W = X \underset{Z}{\times}X = \{(x,x') \in X \times X : \pi(x) = \pi(x')\}$, with the diagonal action $(T \times T)(x,x') = (Tx,Tx')$. We call $W$ the {\em relative product} of $X$ over $Z$. 
We then have:

\begin{prop}\label{iff}
A system \xt\ is an isomorphic extension of \zs\ if and only if there exists a \tl\ factor map 
$\pi : (X,T) \to (Z,S)$ such that the associated relative product $(W, T \times T)$ is uniquely ergodic.
\end{prop}

\begin{proof}
$(\Leftarrow)$:\ First observe that, 
since $(X,T)$ is a factor of $(W,T)$, it is uniquely ergodic. 
As usual, we will denote the invariant measure on $X$ by $\mu$ and the one on $Z$ by $\nu$.
Next observe that on $W$ there are always two obvious $(T\times T)$-invariant measures,
namely the diagonal measure $\mu_\Del = J_*(\mu)$, where $J : X \to W$ is the map
$J(x) = (x,x)$, and the relative product measure $\mu\underset{\nu}{\times} \mu$,
which is obtained as follows:
Let $\mu = \int \mu_z \, d\nu(z)$ be the disintegration of $\mu$ over $\nu$, then
$$
\mu\underset{\nu}{\times} \mu
: = \int (\mu_z \times \mu_z) \, d\nu(z).
$$
Now the fact that $(W,T)$ is uniquely ergodic implies that $\mu_\Del = 
\mu\underset{\nu}{\times} \mu$, i.e. disintegrating the left hand side over $\nu$, we get
$$
\int (\mu_\Del)_z \, d\nu(z) =  \int (\mu_z \times \mu_z) \, d\nu(z).
$$
By uniqueness of the disintegration, we have 
$(\mu_\Del)_z = \mu_z\times\mu_z$ 
for $\nu$-a.e. $z$.

Clearly, $\nu$-a.e. $(\mu_\Del)_z$ is supported on the diagonal set $\Del_z = \{(x,x):x\in\pi^{-1}(z)\}$,
while the only measures on $\pi^{-1}(z)$ whose Cartesian square is supported by this diagonal set are point masses. 
We deduce that $\nu$-a.e. $\mu_z$ is a point mass.
Denoting by $\varphi(z)\in \pi^{-1}(z)$ the atom of $\mu_z$, we easily see that $\varphi:Z\to X$
is a measure-theoretic isomorphism whose inverse (wherever defined) coincides with $\pi$.

$(\imp)$:\ 
We now assume that \xt\ is an isomorphic extension of \zs, i.e., that there exists a $\nu$-a.e. defined measurable bijection $\varphi:Z\to X$, whose inverse coincides with $\pi$ (i.e., $\varphi(z)\in\pi^{-1}(z)$, $\nu$-a.s.), and such that $\mu = \varphi_*(\nu)$ is a unique \im\ on $X$.
Clearly, the disintegration of $\mu$ over $\nu$ is into point masses on the graph of $\varphi$:
$$
\mu =\int\delta_{\varphi(z)}\,d\nu(z).
$$

Let now $\lambda$ be a $T\times T$-\im\ on $W$. It disintegrates over $\nu$ as
$$
\lambda = \int\lambda_z\,d\nu(z),
$$
where $\nu$-a.e. $\lambda_z$ is supported by the fiber of $z$ in $W$, i.e., by $\pi^{-1}(z)\times\pi^{-1}(z)$. Letting $\lambda_{i,z}$ $(i = 1,2)$ denote
the marginals of $\lambda_z$ on the two copies of $\pi^{-1}(z)$, respectively, we
can see that both measures $\int\lambda_{i,z}\,d\nu(z)$ are \im s on $X$. So, both of them must
equal $\mu$ and, by uniqueness of the disintegration, for $\nu$-a.e. $z$ we have 
$$
\lambda_{1,z} = \lambda_{2,z} = \delta_{\varphi(z)},
$$
hence $\lambda_z = \delta_{\varphi(z)}\times \delta_{\varphi(z)} = \delta_{(\varphi(z),\varphi(z))}$.
This proves uniqueness of $\lambda$.
\end{proof}

\br
\section{A generic extension in $\overline\cob(\Gcal)$ is isomorphic}
Our purpose in this section is to show the existence of isomorphic extensions which are not almost one-to-one. Given a strictly ergodic \tl\ system \zs, we will investigate 
a certain associated class of skew product extensions,
with an appropriate fiber space $Y$.

Let $Y$ be a (sufficiently rich) compact metric space, and let $\Gcal$ be a
closed subgroup of the group of all homeomorphisms of $Y$, equipped with the uniform metric 
(this topology makes $\Gcal$ a Polish topological group). Let $C(Z,\Gcal)$ be the family of all continuous maps $G:Z\to\Gcal$, which we will call \emph{cocycles}. We will write $G_z$ (rather than $G(z)$) to denote the homeomorphism of $Y$ associated to $z$, while $G_z(y)$ will denote its value at $y\in Y$. Equipped with the uniform topology, $C(Z,\Gcal)$ is again a Polish group, with multiplication and inverse defined poinwise:
$(GH)_z = G_zH_z$, $(G^{-1})_z = (G_z)^{-1}$. 

From a cocycle $G\in C(Z,\Gcal)$ we can create a \emph{coboundary} $G_S^{-1}G$, where $G_S$ is defined by $(G_S)_z = G_{Sz}$ (note that $(G_S)^{-1} = (G^{-1})_S$ hence one can skip the parentheses). 
Thus we have $(G_S^{-1}G)_z(y) = G_{Sz}^{-1}(G_z(y))$.
Clearly, the coboundary still belongs to $C(Z,\Gcal)$. The collection of all coboundaries obtained in this manner will be denoted by $\cob(\Gcal)$. In general, this is neither a subgroup nor a closed subset of $C(Z,\Gcal)$, 
so we will work with the closure $\overline\cob(\Gcal)$, which is just a Polish space (enough to use category arguments).
\smallskip

With each cocycle $G\in C(Z,\Gcal)$ we associate a \emph{skew product extension} of $(Z,S)$ defined on
the product space $X = Z \times Y$ by:
$$
T_G(z,y) = (Sz, G_z(y)).
$$

We are now in a position to formulate the main theorem of this section:
\begin{thm}\label{main}
Let $(Z,S)$ be an (infinite) strictly ergodic \tl\ \ds, let $\Gcal$ be a pathwise connected subgroup of the group of all homeomorphisms of $Y$, with the following 
 property:
 \begin{quote}
(A) For every nonempty open set $V$ in $Y$ and $\ep > 0$ there are homeomorphisms $h_1, h_2, \dots, h_M$ in 
$\Gcal$ such that
$$
\frac{1}{M} \sum_{j=1}^M \ch_{h_j(Y \setminus V)}(y) \le \ep,
$$
for all $y \in Y$ ($\ch_F$ denotes the characteristic function of a set $F$).
\end{quote} 
Then, for a generic (i.e., for a member of a residual subset)
cocycle $G\in \overline\cob(\Gcal)$, the corresponding extension $(X,T_G)$ of \zs\ is both minimal and isomorphic.
\end{thm}

\begin{rems}
\begin{enumerate}
	\item The assumptions imply that $Y$ contains nontrivial pathwise connected components, hence is uncountable.
	\item The most obvious example of a space $Y$ and a group $\Gcal$ satisfying these assumptions is the
	circle and the group of all orientation-preserving circle homeomorphisms. 
	It is easy to see that the unit interval, even with its entire group 
	of orientation preserving homeomorphisms, does not satisfy the condition ($A$).
\item 
Clearly, the generic skew product, whose existence is claimed in this theorem, is an isomorphic but not an almost one-to-one extension of \zs. 
There exists a function $\varphi: Z \to Y$ such that $z\mapsto (z,\varphi(z))$ is 
an isomorphism from \zs\ to \xt, and the unique \im\ $\mu$ on $X$ is supported by the graph of $\varphi$. By minimality, this graph must be dense in $X = Z \times Y$.
\item 
Applied to minimal equicontinous systems \zs, Theorem \ref{main} answers negatively (although ineffectively) the Question \ref{que1} and hence also \ref{que}. 
\end{enumerate}
\end{rems}

\begin{proof} We need to demonstrate two claims: genericity of minimality and genericity of being 
an isomorphic extension. Of course, the intersection of two generic properties is still generic.
Now, the first claim is proven in \cite[Theorem 1]{GW-79} (notice that the condition $(A)$ implies that $\Gcal$ acts minimally on $Y$). Alternatively, one can use \cite[Theorem 2]{GW-79} (strict ergodicity includes minimality). Thus, it remains to show that isomorphic extensions form a 
residual subset of $\overline\cob(\Gcal)$.
\smallskip

We begin with an outline of the proof. We will consider the relative product 
$X \underset{Z}{\times}X = \{(x,x') \in X \times X : \pi(x) = \pi(x')\}$, which in our case --- where
$X = Z \times Y$ is a product space --- is homeomorphic to
$\widetilde{X} : = Z \times Y \times Y$.
Then, given $G\in C(Z,\Gcal)$, 
we define the corresponding cocycle extension as the map:
$$
\widetilde T_G(z,y_1,y_2) = (Sz, G_z(y_1), G_z(y_2)).
$$
We will show that for a generic $G\in\overline\cob(\Gcal)$ this defines a uniquely ergodic
map on $\widetilde{X}$. 
This will be done by a modification of the proof of \cite[Theorem 2]{GW-79}, and,
in view of Proposition \ref{iff} this will complete our proof.
\medskip

We begin with a lemma, where the technical assumptions made on $\Gcal$ are essential.

\begin{lem}\label{la}
Let $V$ be a nonempty open subset of $Y$, and fix $\ga>0$. There exists a continuous map $t \mapsto h_t$ 
from $[0,1]$ into $\Gcal$ such that for all $y \in Y$, with $\la$ denoting the Lebesgue measure 
on $[0,1]$, we have
$$
\la \{t \in [0,1] : h^{-1}_t(y) \notin V\} < \ga.
$$
\end{lem}

\begin{proof}
By assumption, there exist $\tilde{h}_1, \tilde{h}_2, \dots,\tilde{h}_M \in \Gcal$
satisfying the condition 
$$
\frac1M \sum_{i=1}^M \ch_{\tilde{h}_i(Y\setminus V)}(y) \le \frac{\ga}{2}.
$$

Define $h_t = \tilde{h}_i$ for $t\in[\frac{i - 1}M +  \frac\gamma{4M}, \frac iM - \frac\gamma{4M}] \ (1  \le  i  \le M)$ 
and extend the map $t \mapsto h_t$, continuously to all of $[0,1]$ (here we use pathwise connectedness of $\Gcal$). Now consider two cases for $y\in Y$:
\begin{itemize}
	\item If $y \not \in \bigcup_{i=1}^M \tilde{h}_i(Y \setminus V)$, then, for every
$t \in \bigcup_{i=1}^M [\frac{i - 1}M +  \frac\gamma{4M}, \frac iM - \frac\gamma{4M}]$, we have
$h^{-1}_t(y) \in V$ and
$$
\la \{t \in I : h^{-1}_t(y) \notin V\}  \le 2M \frac{\gamma}{4M} = \frac{\ga}{2}.
$$

\item Otherwise there is a nonempty set $F\subset\{1,\dots,M\}$ with 
$y \in \bigcap_{i\in F} \tilde{h}_i(Y \setminus V)$, and then
$$
\frac {|F|}M = \frac1M \sum_{i=1}^M \ch_{\tilde h_i(Y\setminus V)}(y) \le \frac{\ga}2,
$$
implying 
$$
\la \{t \in I : h^{-1}_t(y) \notin V\}  \le \frac{|F|}{M} + 2M \frac{\gamma}{4M} \le \ga.
$$
\end{itemize}
\end{proof}

We proceed with the main proof. For $f \in C(\widetilde X,[0,1])$ and $\ep > 0$ we denote
$$
E_{f,\ep} = \Bigl\{G \in C(Z,\Gcal) : \exists c\in \R,\exists n\ge 1, \Bigl\| \frac{1}{n+1} 
\sum_{k =0}^n f(\widetilde T_G^k(z,y_1,y_2)) - c \, \Bigr\| < \ep\Bigr\}.
$$
Clearly $E_{f,\ep}$ is open in $C(Z,\Gcal)$ and it is easy to check that $\Rcal = \bigcap_{i,j\in\N} E_{f_i,1/j}$, where $\{f_i\}$ is a countable dense subset of $C(\widetilde X,[0,1])$, consists precisely 
of the cocycles $G\in C(Z,\Gcal)$ for which $(\widetilde X,\widetilde T_G)$ is uniquely ergodic\footnote{Formally, to guarantee unique ergodicity, in the condition defining $E_{f,\ep}$, in place of $\exists n\ge 1$ we should demand $\exists n_0\ge 1\, \forall n\ge n_0$. However, it is not hard to see, that if one such $n$ exists, then an $n_0$, satisfying the correct condition, with a slightly larger $\ep$, can also be found 
(much larger than~$n$).}. Thus, all we need to show is that for an arbitrary $f \in C(\widetilde X,[0,1])$ and $\ep >0$, $E_{f,\ep}\cap \overline\cob(\Gcal)$ is dense in $\overline\cob(\Gcal)$, i.e., that $\cob(\Gcal)\subset\overline{E_{f,\ep}}$. 

Let $H_S^{-1} H$ be a coboundary. We have 
\begin{equation}\label{totu}
H_S^{-1} H\in\overline{E_{f,\ep}} \iff \Id\in H_S \overline{E_{f,\ep}} H^{-1} = \overline{H_S E_{f,\ep} H^{-1}},
\end{equation}
where $\Id\in C(Z,\Gcal)$ assigns to each $z\in Z$ the identity map $\Id_Y$ on $Y$ (the last equality
follows from the fact that left and right multiplications by a fixed cocycle are homeomorphisms of
$C(Z,\Gcal)$). We need to identify the latter set.

Define $\psi_H(z,y_1,y_2)=(z,H_z^{-1}(y_1),H_z^{-1}(y_2))$. This is a self-homeo\-morphism of 
$\widetilde X$. Thus, we have the obvious equality of uniform norms:
$$
\Bigl\| \frac{1}{n+1}\sum_{k =0}^n f(\widetilde T_G^k(z,y_1,y_2)) - c \, \Bigr\| = \Bigl\| \frac{1}{n+1} \sum_{k =0}^n f(\widetilde T_G^k(\psi_H(z,y_1,y_2))) - c \, \Bigr\|,
$$
i.e., in the definition of $E_{f,\ep}$, we can replace the argument $(z,y_1,y_2)$ by $\psi_H(z,y_1,y_2)$.

A tedious but straightforward computation shows that 
$$
\widetilde T^k_{H_S^{-1}GH}(\psi_H(z,y_1,y_2)) = \psi_H(\widetilde T^k_G(z,y_1,y_2)).
$$
Combined with the preceding observation, this easily implies that 
$$
H_S^{-1}GH \in E_{f,\ep} \Leftrightarrow G\in E_{f\circ\psi_H,\ep}, \ {\text{ i.e.,}} \quad   
H_SE_{f,\ep}H^{-1} = E_{f\circ\psi_H,\ep}
$$ 
(whence the closures are equal).

The equivalence \eqref{totu} now means that for the required inclusion $\cob(\Gcal)\subset\overline{E_{f,\ep}}$ it suffices to show that $\id\in \overline{E_{f\circ\psi_H,\ep}}$,
for any $\ep>0$, any function $f\in C(\widetilde X, [0,1])$,  and any cocycle $H$. 
Since the family $\{f\circ\psi_H: f\in C(\widetilde X,[0,1]), H\in C(Z,\Gcal)\}$ equals $C(\widetilde X,[0,1])$, what we need is that for any $f\in C(\widetilde X,[0,1])$ and $\ep>0$, we have $\Id \in \overline{E_{f,\ep}}$. We will show a stronger fact: $\id$ can be uniformly approximated by 
{\sl coboundaries} from $E_{f,\ep}$. We formulate this as a lemma:

\begin{lem}\label{sig}
Given $f\in C(\widetilde X,[0,1])$, $\ep>0$ and $\del > 0$, there exists a coboundary $G=H_S^{-1}H \in E_{f,\ep}$ satisfying $d(G,\Id) < \del$, (here $d$ denotes the uniform distance on $C(Z,\Gcal)$).
\end{lem}

\begin{proof}
Observe that the expression in the (modified) definition of $E_{f,\ep}$, 
$$
\frac1{n+1}\sum_{k=0}^n f(\widetilde T_G^k(\psi_H(z,y_1,y_2)) = \frac1{n+1}\sum_{k=0}^n f(S^kz,H^{-1}_{S^kz}(y_1),H^{-1}_{S^kz}(y_2))
$$
equals the ergodic average, under the action of $S$, of the real-valued function (of one variable, with two parameters) $f_{y_1,y_2}\in C(Z,[0,1])$ given by
$$
f_{y_1,y_2}(z) = f(z,H^{-1}_z(y_1),H^{-1}_z(y_2)).
$$
Strict ergodicity of \zs\ implies that if we fix $(y_1,y_2)\in Y^2$, then
these averages tend uniformly over $z\in Z$ to the integral
$$
\int f_{y_1,y_2} \, d\nu(z)
$$
(recall that $\nu$ is the unique invariant measure on $Z$). Since the map $(y_1,y_2)\mapsto f_{y_1,y_2}$ is obviously continuous, the family of functions $\{f_{y_1,y_2}:(y_1,y_2)\in Y^2\}$ is compact 
and thus the above convergence is uniform also over $(y_1,y_2) \in Y^2$. 

So, the condition $G\in E_{f,\ep}$ in the lemma is equivalent to 
\begin{equation}\label{war}
\left | \int f(z, H^{-1}_z(y_1), H^{-1}_z(y_2))\, d\nu(z) - c\, \right| < \ep,
\end{equation}
for some constant $c$ and all $(y_1,y_2) \in Y^2$.

We let $V\subset Y$ be an open set such that 
$$
\sup | f(z, v_1, v_2) - f(z, v'_1,v'_2) | < \frac\ep2,
$$
where the supremum ranges over all $z\in Z$ and $v_1, v_2, v'_1, v'_2 \in V$, and set $\gamma = \frac\ep{16}$. Lemma \ref{la} provides a continuous assignment $t\mapsto h_t$ with the appropriate properties.

We can now proceed with the construction of $H$ (and thus $G$). Let $\eta >0$ be such that 
$$
|t_1 - t_2 | < \eta \implies d(h^{-1}_{t_1}h_{t_2}, \id_Y) < \del
$$ 
and fix an $N \in \N$ with $\frac1N < \min\{\frac\eta2,\ga\}$.
There exists a measurable subset $K$ of $Z$ with $\nu(K)>0$, such that the images
$K, S(K), \dots, S^{N^2 -1}(K)$ are pairwise disjoint while the measure of their union exceeds $1- \ga$.
By regularity, we can assume that $K$ is a closed set, moreover, we can arrange that it is homeomorphic 
to the Cantor set. Let $\tilde{\theta} : K \to [0,1]$ be a continuous surjection for which the push-forward measure 
$\tilde{\theta}_*(\nu|K)$ equals $\la$, where $\nu|K$ denotes the (normalized) conditional measure on $K$ obtained from $\nu$ (such a surjection exists for any continuous probability measure on the Cantor set). Extend $\tilde{\theta}$ to $\bigcup_{i=0}^{N^2 -1} S^i(K)$ by:
$$
\tilde{\theta}(z) = \tilde{\theta}(S^{-i}z) {\text{\ \ if\ \ }} \ z \in S^i(K)\ \ (i=1,\dots,N^2-1)
$$
and then extend it again to a continuous map $\tilde{\theta}: Z \to [0,1]$. Finally, for every $z\in Z$, put
$$
\theta(z) = \frac1N \sum_{i=0}^{N-1} \tilde{\theta}(S^{-i}z)
$$
and define $H\in C(Z,\Gcal)$ by setting $H_z = h_{\theta(z)}$. It remains to verify the required 
properties of $G = H_S^{-1}H$.

We claim that $G\in E_{f,\ep}$, namely, that for every $y_1, y_2 \in Y$, we have
$$
\left | \int f(z, H_z^{-1}(y_1), H_z^{-1}(y_2))\, d\nu(z) - c\,\right | < \ep
$$
(i.e., \eqref{war} holds), where $c = \int f(z, v, v)\,d\nu(z)$, with $v$ being any point selected from the open set $V\subset Y$. 
Indeed,

\begin{align*}
 & \int | f(z, H_z^{-1}(y_1), H_z^{-1}(y_2)) -  f(z,v,v) | \, d\nu(z) \le\\
 & \int_{\bigcup_{i=0}^{N^2-1} S^i (K)}  | f(z, H_z^{-1}(y_1), H_z^{-1}(y_2)) -  f(z,v,v) | \, d\nu(z) +
 2\|f\|\gamma\le \\
 & \int_{\bigcup_{i=0}^{N^2-N-1} S^i(K)}  | f(z, H_z^{-1}(y_1), H_z^{-1}(y_2)) - f(z,v,v)| \, d\nu(z) +2\|f\|(\gamma+N\cdot\nu(K))< \\ 
  & \sum_{i=0}^{N^2-N-1} \int_K  | f(S^iz, H_{S^iz}^{-1}(y_1), H_{S^iz}^{-1}(y_2)) -  f(z,v,v) | \, d\nu(z) + 4\gamma
=  \\
 & \sum_{i=0}^{N^2-N-1} \int_K  | f(S^iz, H_z^{-1}(y_1), H_z^{-1}(y_2)) -  f(z,v,v) | \, d\nu(z) + 4\gamma,   
\end{align*}
(we have used the fact that for $z\in K$, $H_z=H_{Sz} = \cdots = H_{S^{N^2-N-1}z}$, and the inequalities $\nu(K)\le \frac1{N^2}, \frac1N<\gamma$ and $\|f\|\le 1$). Next, consider a point $z\in K$ for which both $h^{-1}_{\theta(z)}(y_1)\in V$ and $h^{-1}_{\theta(z)}(y_2)\in V$. By the choice of $V$, for such $z$, the integrand does not exceed $\frac\ep2$. On the other hand, a point $z\in K$ does not fulfill this condition only when $t=\theta(z)$ is such that either $y_1$ or $y_2$ fall outside $h_t(V)$. The measure $\lambda$ of such $t$'s is at most $2\gamma$ (this is what Lemma \ref{la} yields), i.e., the conditional measure $\nu|K$ of such $z$'s is at most $2\gamma$. We conclude that the last line above is bounded by
$$
(N^2 - N)\nu(K) (\frac\ep2 + 4\gamma) + 4\gamma \le \frac\ep2+8\gamma\le\ep,
$$
as claimed (we have used one more time the fact that $\nu(K) \le \frac1{N^2}$).
\smallskip

Finally, in order to show that $d(H^{-1}_{Sz}H_z,\Id_Y) < \del$, for every $z\in Z$, note that,
because $\theta(z)$ is defined as an $N$-step ergodic average, we have 
$|\theta(Sz) - \theta(z)| \le \frac2N < \eta$, which, by the choice of $\eta$, yields 
$d(H^{-1}_{Sz} H_z, \id_Y) = d(h^{-1}_{\theta(Sz)} h_{\theta(z)}, \id_Y)<\delta$.
This concludes the proof of Lemma \ref{sig}.
\end{proof}
The proof of Theorem \ref{main} is now complete.
\end{proof}

\begin{rmk}
The same proof will show that for a residual set of $G$'s in $\overline\cob(\Gcal)$,
the analogous relative product action on
$Z \times Y \times Y \times \cdots \times Y$ ($k$ times)
is uniquely ergodic, and hence also that the infinite
product will be uniquely ergodic for a residual set of $G$'s.
\end{rmk}

\br

\section{Some concrete examples and related results}

As in \cite{GW-79} we can single out the following concrete examples.
Let $(Y, \Gcal)$ denote $(\mathbb{P}^n, SL(n+1,\R))$ or $(Q, \Gcal)$ where 
$\Gcal$ is the identity path component of $\Hcal(Q)$, the group of homeomorphisms of $Q$. 
Here $\mathbb{P}^n$ and $Q$ denote the $n$-dimensional 
projective space and the Hilbert cube respectively. 
It is easy to check that, in both cases, the conditions in Theorem \ref{main} hold (see \cite{GW-79}). 
Clearly the action of $SL(n + 1,\R)$ on $\mathbb{P}^n$ is (literally) transitive and 
this is also true for the action of $\Gcal$ on $Q$; in particular these actions are minimal. 
Thus for an arbitrary infinite, minimal, uniquely ergodic, metric system 
$(Z, S)$, for a residual set of 
$G$'s in $\overline\cob(\Gcal)$, the corresponding  
homeomorphisms of $X = Z \times Y$, where
$Y = \mathbb{P}^n$  or $Q$, is strictly ergodic. Moreover, 
denoting by $\mu$ the unique invariant measure on $X$ and by $\nu$ the unique invariant
measure on $Z$, by Theorem \ref{main}, the extension map 
$\pi : (X,\mu,T_G) \to (Z,\nu,S)$ is a measure-theoretic isomorphism;
i.e. $\pi$ is an isomorphic extension.
In particular if we let $(Z, S) = (\T, R_\al)$, an irrational rotation of the circle, and 
$(Y, \Gcal)= (\mathbb{P}^1,SL(2, \R))$, then, since $\mathbb{P}^1$ is homeomorphic to $\T$, we can
obtain these minimal systems on the torus $\T^2$.

\br

Recall that a metric dynamical system is called {\em tame} when its enveloping semigroup has 
cardinality $\le 2^{\aleph_0}$. Both $(Z,S) =(\T,R_\al)$ and $(Y,\Gcal) =(\mathbb{P}^n, SL(n+1,\R))$
are tame (for more details see \cite{Ellis-93}, \cite{Ak-98} and \cite[Example 8.31(6)]{GM-14}). 
However, none of the minimal systems $(X,T_G) = (Z \times Y, T_G)$ given by Theorem \ref{main} is tame. 
In fact, a theorem of Huang \cite{Huang}, Kerr \& Li \cite{KL}, and Glasner \cite{Gl-str}
asserts that a tame minimal dynamical system is necessarily {\em almost automorphic},
i.e., an almost one-to-one extension of its maximal equicontinuous factor, and that this extension is isomorphic. More precisely, as quoted from \cite[Theorem 5.1]{Gl-str}, we have:

\begin{thm}\label{almost-auto}
Let $\Ga$ be an Abelian group and $(X,\Ga)$ a metric tame minimal system. 
Then:
\begin{enumerate}
\item
The system $(X,\Ga)$ is almost automorphic. Thus there exist:
\begin{enumerate}
\item
A compact topological group $Y$ with Haar measure $\eta$,
and a group homomorphism
$\kappa:\Ga \to Y$ with dense image.
\item
A homomorphism $\pi:(X,\Ga) \to (Y,\Ga)$,
where the $\Ga$ action on $Y$ is via $\kappa$.
\item
The sets $X_0= \{x\in X: \pi^{-1}(\pi(x)) = \{x\}\}$
and $Y_0=\pi(X_0)$ are dense $G_\del$ subsets of
$X$ and $Y$ respectively.
\end{enumerate}
\item
The system $(X,\Ga)$ is uniquely ergodic with unique invariant measure $\mu$
such that with $\pi_*(\mu)=\eta$, the map $\pi:(X,\mu,\Ga) \to (Y,\eta,\Ga)$ is a
measure-theoretic isomorphism of the corresponding measure preserving
systems.
\end{enumerate}
\end{thm}

\br

By a theorem of Furstenberg and Kesten \cite{Fur-K}, for an ergodic system
$(Z,\nu,S)$ and a 
measurable map (cocycle) $\rho : Z \to SL(n,\R)$,
such that both $\log^+\|\rho(z)\|$ and $\log^+\|\rho^{-1}(z)\|$ are integrable functions, the
{\em Lyaponov} limit
\begin{equation}\label{Ly}
\La(\rho) = \lim \frac{1}{n+1} \log \|\rho(T^nz)\cdots\rho(Tz)\rho(z)\|
\end{equation}
exists $\nu$ almost everywhere.
In the case where $(Z,S,\nu)$ is a strictly ergodic system and $\La(\rho) = 0$,
it was shown by Furman \cite{Fu} that the limit exists uniformly on $Z$. 
However, in general, even in the case where $(Z,S,\nu)$ is strictly ergodic,
as was shown by Walters \cite{W} and Herman \cite{H}, the convergence need not be uniform.

In the two-dimensional case; i.e. when $\rho$ maps into $SL(2,\R)$, Furman considers the associated skew product $X = Z \underset{\rho}{\times} Y$, where $Y = \mathbb{P}^1$ and then,
using Oseledec' theorem \cite{Os}, he proves the following:

\begin{thm}
Let $(Z,S,\nu)$ be a strictly ergodic system, $\rho : Z \to SL(2,\R)$ a continuous cocycle, and
let $(X,T)= (Z \underset{\rho}{\times} \mathbb{P}^1,T)$ be the corresponding skew-product.
Then there are three possibilities:
\begin{enumerate}
\item[(1)]
$\La(T) : = \La(\rho) = 0$, in which case the convergence in (\ref{Ly}) is uniform
and $(X,T)$ is uniquely ergodic.
\item[(2)]
$\La(\rho) > 0$.
There are on $X$ exactly two ergodic measures $\theta^+$ and $\theta^-$. 
They are graph measures, i.e. carried by
the graphs of two measurable functions $u^\pm : Z \to \mathbb{P}^1$.
\begin{enumerate}
\item[(2a)]
On $(X,T)$ there is a unique minimal set $M$, with 
 $M= \supp \theta^+ = \supp \theta^-$,
 and there is a dense $G_\del$ subset $Z_0 \subset Z$ on which 
\begin{equation*}
\liminf \frac{1}{n+1} \log \|\rho(T^nz)\cdots\rho(Tz)\rho(z)\| < \La(\rho).
\end{equation*}

\item[(2b)]
On $X$ there are exactly two minimal sets $M^+$
and $M^- $, which are the graphs of 
continuous functions $u^\pm : Z \to \mathbb{P}^1$.
Each of these minimal sets supports a unique ergodic measure,
$\theta^+$ and $\theta^-$ with 
$ M^+ =  \supp \theta^+$ and $M^- =  \supp \theta^-$,
and again the convergence in (\ref{Ly}) is uniform.
\end{enumerate}
\end{enumerate}
\end{thm}

Examples of the type (2a) were given by Walters \cite{W} and Herman \cite{H}.
Furman also shows that a cocycle $\rho$ of type (2b) is continuously diagonalizable and
continuously cohomologous to
a function whose values eventually lie in the set of positive matrices.
For more details see \cite{Fu}.
It follows from this theorem that the strictly ergodic systems on $Z \times \mathbb{P}^1$
whose existence is shown in \cite[Theorem 2]{GW-79} and in Theorem \ref{main} above
are all of the type $\La(T) =0$.

\br

\end{document}